\documentclass[preprint,12pt]{elsarticle}

\usepackage{amssymb}
\usepackage{amsmath}
\usepackage{amsthm}

\newtheorem{theorem}{Theorem}
\newtheorem{lemma}{Lemma}

\newtheorem{definition}{Definition}
\newtheorem{remark}{Remark}

\begin{document}

\begin{frontmatter}



\title{Countable periodic solutions of the Lorentz force equation under a time-dependent current}


%\author{} %% Author name

%\affiliation{organization={},%Department and Organization

\author[label1,label3]{Ka Xie}
\ead{xieka@amss.ac.cn}

\author[label1,label3]{Pengcheng Xu\corref{cor1}}
\ead{xupc@amss.ac.cn}
\cortext[cor1]{Corresponding author}
           
\author[label1,label3,label4]{Zuohuan Zheng}
\ead{zhzheng@amt.ac.cn}

%Academy of Mathematics and System Sciences, Chinese Academy of Sciences,  , 
\affiliation[label1]{organization={Academy of Mathematics and Systems Science, Chinese Academy of Sciences},
            addressline={55 Zhongguancun East Road, Haidian District},
            city={Beijing},
            postcode={100190},
            %state={},
            country={China}}

\affiliation[label3]
            {organization={University of Chinese Academy of Sciences},
            addressline={19 Yuquan Road, Shijingshan District},
            city={Beijing},
            postcode={100049},
            %state={},
            country={China}}  

\affiliation[label4]
            {organization={Chinese Academy of Sciences' Reliability Assurance Center},
            addressline={9 Dengzhuang South Road, Haidian District},
            city={Beijing},
            postcode={100094},
            %state={},
            country={China}}              

%% Abstract
\begin{abstract}
The resonant dynamics of a charged particle, governed by the Lorentz force equation in an electromagnetic field generated by a current-carrying wire with a small harmonic modulation, is considered in this study. When regarded as a Hamiltonian system with periodic perturbation, the resonance of periodic orbits in the unperturbed system is analyzed by the Melnikov method. The existence of exactly one harmonic radial periodic solution with period $T_1$ is confirmed, matching the period of the current. Moreover, it is established that any other radial periodic solution must be subharmonic with period $nT_1$ for some integer $n > 1$, with at most one such solution for each $n$. Dynamically, these surviving periodic orbits correspond to invariant cylinders that partition the phase space and globally confine the particle's radial motion.

\end{abstract}



%% Keywords
\begin{keyword}
Lorentz force equation; Electromagnetic field; Straight wire; Oscillating current; Radially periodic solution; Melnikov method


\MSC[]{34C25 \sep 78A35 \sep 83A05 \sep 34E10 \sep 34D30}

\end{keyword}

\end{frontmatter}



\section{Introduction}
The Lorentz force equation, which describes the motion of a charged particle slowly accelerated in an electromagnetic field, is one of the fundamental equations in Mathematical Physics. This equation, historically formulated by Planck\cite{planck1906prinzip} and Poincaré\cite{Poincar1906SurLD} in 1906, remains central to the study of particle dynamics in various physical systems, ranging from plasma physics to astrophysical phenomena. Classic treatments of the equation can be found in numerous textbooks, such as \cite{griffiths2023introduction}, \cite{jackson1999classical}, and \cite{landau2013classical}. Despite its long-standing history, the qualitative and quantitative understanding of solutions to the Lorentz force equation has been sorely lacking until recently. The existence of solutions under Dirichlet and Neumann conditions is proved by the topological degree method \cite{bereanu2008boundary}. A systematic investigation of the Lorentz force equation has recently commenced, employing the tools of critical point theory. More precisely, several existence and multiplicity results for the Lorentz force equation subject to periodic or Dirichlet boundary conditions have been obtained by \cite{arcoya2019critical} \cite{arcoya2020lusternik} using rigorous critical point theory.

A special solution, known as the periodic solution, has attracted significant attention in the study of the Lorentz force equation due to their connection to recurring and stable behaviors in electromagnetic fields. The existence of periodic solutions for the Lorentz force equation has been researched by many authors. The existence of periodic solutions for a long range of potentials that permit isolated singularities has been established\cite{garzon2020periodic}. This is achieved through a topological approach, which encompasses significant physical cases, such as the Coulomb potential and the magnetic dipole. The existence of $T$-periodic solutions for the Lorentz force equation, in the relevant physical configuration where the electric field exhibits a singularity at zero, is established by using Szulkin’s critical point theory \cite{szulkin1986minimax}, providing further insight into the dynamics of charged particles in singular electromagnetic fields\cite{Arcoya2023}. In a recent development, Alberto Boscaggin demonstrates the existence of infinitely many $T$-periodic solutions for the Lorentz force equation in the context of a singular electric field $E$\cite{boscaggin2024infinitely}. This proof employs a min-max principle of the Lusternik-Schnirelmann type within the framework of non-smooth critical point theory. The findings have practical applications, including the motion of a charged particle under the action of a Liénard-Wiechert potential and the relativistic forced Kepler problem. Furthermore, it is important to highlight that the application of variational methods to singular problems remains an open area of research.

The dynamics of the Lorentz force equation become interesting for a charged particle in an electromagnetic field induced by a time-dependent current along an infinitely long and infinitely thin straight wire recently. The corresponding non-relativistic dynamics have been studied extensively in \cite{aguirre2010motion}, \cite{gascon2004motion},  \cite{gascon2005some}, together with other autonomous wire distributions with symmetries. When the current in an electrically neutral wire is constant, the electric field vanishes, and the magnetic field can be obtained via the Biot-Savart law, reducing it from Maxwell's equations in a rigorous way. However, for a time-dependent current, the regime is no longer magnetostatic, and the Biot–Savart law does not hold in the non-stationary case. Hence it is necessary to deduce the electromagnetic field by solving Maxwell's equations. For compactly supported current distributions, solutions to the Maxwell's equations can be obtained by introducing retarded potentials. However, in this case the infinite wire does not have a compact support. Nevertheless, it is proven that the corresponding retarded potential still gives a solution of Maxwell’s equations at least in the distributional sense\cite{garzon2021motions}.

Within the framework of this model, radial periodic solutions form a primary focus of interest. This type of motion describes a fundamental type of bounded periodic motion for charged particles in electromagnetic fields. Such periodic behavior can model a range of important physical phenomena. For instance, in plasma confinement devices, charged particles exhibit transverse oscillations while streaming along magnetic field lines. Similarly, in particle accelerators, the time-varying magnetic fields used to focus beams and enhance stability can give rise to periodic motion in the transverse plane. Moreover, in certain astrophysical systems, such as simplified models of charged particle dynamics in the magnetospheres of pulsars or around current-carrying cosmic strings, similar radial periodic motions can occur under the influence of complex electromagnetic fields. In all of the above systems, identifying and characterizing radial periodic solutions provides a foundational key to understanding complex charged particle dynamics in oscillatory or non-uniform electromagnetic environments. The existence of radial periodic solutions within an explicit interval of the perturbation parameter $k$ is proved by \cite{garzon2023periodic} using the topological degree method. Concretely, the system has cylindrical symmetry, and the corresponding linear and angular relative momentum are conserved. Thus, the Lorentz force equation is reduced to a planar Hamiltonian system with one degree of freedom. The existence of radially periodic solutions is then proved by using the topological degree, where the interval of the perturbation parameter $k$ depends on the values of the current and the conserved momenta. 

However, several questions remain to be addressed. One question is, when the current is a harmonic function, how many harmonic radial periodic solutions exist for the Lorentz force equation with periods matching the current's oscillation? The second question is, is there any radial periodic solution with a different period than the current's, and how many periodic solutions exist in the system? The third question is, where are these periodic solutions located? 

To answer these questions, a constant current with small harmonic modulation is analyzed. In the unperturbed case corresponding to a constant current, a continuous family of radial periodic orbits is shown to exist, whose oscillation frequencies vary monotonically with energy. When harmonic modulation is introduced, the system exhibits rich resonance phenomena. Using Melnikov analysis, it is demonstrated that exactly one harmonic radial periodic solution exists, which lies near the periodic orbit of the unperturbed system and in which the radial oscillation locks to the driving frequency. Furthermore, it is established that any other radial periodic solution must be subharmonic with period $nT_1$, where $n>1$ is an integer and $T_1$ is the current's period, and that for each $n$, at most one such solution exists. These periodic solutions act as dynamical barriers in the extended phase space, preserving a hierarchy of invariant regions where the charged particle remains trapped.

This paper is structured as follows. In Section\ref{sec2}, the preparatory knowledge is presented, and the model is rigorously constructed from physical principles. In Section\ref{sec4}, the unperturbed system of a one-degree-of-freedom Hamiltonian system  equivalent to the dynamics of the Lorentz force equation is analyzed. It is demonstrated that a unique equilibrium exists, forming a global center, and that the periods of periodic solutions around this center increase monotonically with the Hamiltonian energy. Finally, in Section\ref{sec5}, the Melnikov method is employed to establish the existence of a countable number of radial periodic solutions under small harmonic current modulation. These solutions are characterized by periods that are integer multiples of the current's period, with at most one periodic solution corresponding to each integer multiple.

\section{The Lorentz force equation with an infinite wire}\label{sec2}
Without loss of generality, by normalizing both the speed of light in vacuum and the charge-to-mass ratio to 1, the Lorentz force equation in Special Relativity under study is then the following equation,
\begin{equation}\label{eq04}
    \frac{d}{dt} \left( \frac{\dot{q}(t)}{\sqrt{1 - |\dot{q}(t)|^2}} \right) = E(t, q(t)) + \dot{q}(t) \times B(t, q(t)),
\end{equation}
where $q(t)$ is a charged particle, $E$ is the electric field and $B$ is the magnetic field. 

In this paper, the analysis is conducted on an infinitely long and infinitely thin straight wire, which, without loss of generality, is fixed on the $z$-axis. It is assumed that the wire is electrically neutral, i.e., at any given time, each segment of the wire contains the same number of electrons as protons, and therefore the charge density $\rho$ is null. For a fixed period $T_1 > 0$, the Banach space $C^n([0, T_1]; \mathbb{R})$ is defined to contain $T_1$-periodic functions of class $C^n$. Assume that the wire carries a current $I$ of the form 
\begin{equation*}
    I=I_0 + kI_1(t), 
\end{equation*}
where $I_0 > 0$, $k \geq 0$ are constants and $I_1(t)$ satisfies 
\begin{equation}\label{eq05}
    I_1(t) \in C^2([0, T_{1}]; \mathbb{R}),\quad \int_{0}^{T_{1}} I_1(t)  dt = 0.
\end{equation}
Thus, the current density is expressed as a vector distribution $\vec{J} = (0, 0, J)$ where the scalar component $J$ is defined by 
\begin{equation*}
    J(f) = \int_{\mathbb{R}^2} [I_0 + kI_1(t)] f(t, 0, 0, z)\, dt\, dz,
\end{equation*}
for every $f \in \mathcal{D}(\mathbb{R}^4)$, where $\mathcal{D}(\mathbb{R}^4)$ denotes the space of test functions in $\mathbb{R}^4$. Furthermore, introducing the electromagnetic potential $\vec{A}$ and $\Phi$, then the electromagnetic field given by
\begin{equation}\label{eq02}
    B(t, q) = \nabla \times \vec{A}(t, q), \quad
    E(t, q) = -\frac{\partial}{\partial t} \vec{A}(t, q) - \nabla \Phi(t, q),
\end{equation}
solves Maxwell's equations uniquely for a specific distribution of charge and current. The symbols $\nabla$ and $\nabla \times$ represent the gradient and the rotational with respect to $q$, while $\partial_t$ is the time partial derivative. Note that these derivatives are in principle defined in a distributional sense. In this case, by employing the Lorentz gauge condition, Maxwell's equations are decoupled into two wave equations whose solutions are the \textit{delay potentials} $\Phi$ and $\vec{A}$, where the scalar potential $\Phi \equiv 0$ and the vector potential $\vec{A}(t, q)$ satisfies that 
\begin{equation*}
    \vec{A}(t, q) = A(t, r) \mathbf{e}_z, \quad A(t, r) = -\frac{\mu_0}{2\pi} \left[ a_0(r) + k a(t, r) \right], 
\end{equation*}
and
\begin{equation}\label{eq06}
    a_0(r) = I_0 \ln r, \quad a(t, r) = \int_{0}^{\infty} \frac{I_1[t, r, \tau]}{\sqrt{\tau^2 + r^2}} \, d\tau,
\end{equation}
where $ \mu_0 $ is the vacuum permeability constant for unitary light speed, $r$ denotes the radial variable in the $XY$-plane, $\mathbf{e}_z$ denotes the positive unit vector along the $z$ axis, and the bracket $[t, r, \tau] = \left( t - \sqrt{\tau^2 + r^2} \right)$ denotes the delay effect of the potential. It follows from this result that if $I_1(t) \in C^n([0, T_1]; \mathbb{R})$, then $a(t, r) \in C^n([0, T_1] \times \mathbb{R}^+; \mathbb{R})$. More mathematical details can be found in \cite{garzon2021motions}.

To study the Lorentz force equation \eqref{eq04}, Garzón in \cite{garzon2023periodic} introduced cylindrical coordinates \( q = (r \cos \theta, r \sin \theta, z) \) and denoted the relativistic radial, angular, and linear momenta by \( p_r \), \( L \), and \( p_z \), respectively, defined as
\begin{equation*}
\begin{aligned}
    &p_r = \frac{\dot{r}}{\sqrt{1 - \dot{r}^2 - r^2 \dot{\theta}^2 - \dot{z}^2}}, \quad L = \frac{r^2 \dot{\theta}}{\sqrt{1 - \dot{r}^2 - r^2 \dot{\theta}^2 - \dot{z}^2}}, \\
    &p_z = \frac{\dot{z}}{\sqrt{1 - \dot{r}^2 - r^2 \dot{\theta}^2 - \dot{z}^2}} + A(t, r).
\end{aligned}   
\end{equation*}
It is clear that the motion of a charged particle is described by its radial component, depending implicitly on $L$, $p_z$ and $k$. In particular, this allows for the reduction of \eqref{eq04} to a planar, single-degree-of-freedom Hamiltonian system, namely
\begin{equation}\label{eq07}
\begin{cases}
  & \dot{r}=\frac{p_{r} }{\sqrt{1+(p_{z} -A)^{2}+p_{r} ^{2} +L^{2} r^{-2}    } },   \\
  & \dot{p_{r} }=\frac{L^{2}r^{-3} + (p_{z} -A )\partial_{r} A}{\sqrt{1+(p_{z} -A)^{2}+p_{r} ^{2} +L^{2} r^{-2}    } },
\end{cases}
\end{equation}
which is Hamiltonian for 
\begin{equation*}
    H \left ( t,r, p_{r}\right ) = \sqrt{1+(p_{z} -A)^{2}+p_{r} ^{2} +L^{2} r^{-2}} .
\end{equation*}
Here $L$ and $p_z$ are first integrals of \eqref{eq04} for every $k \geq 0$. More mathematical details can be found in \cite{garzon2023periodic}. 

In the subsequent analysis, the following functions are defined for ease of representation: 
\begin{equation}\label{eq006}
    \begin{aligned}
         &F_1(r, p_r, k)= \frac{p_{r} }{\sqrt{1+(p_{z} -A)^{2}+p_{r} ^{2} +L^{2} r^{-2} } }, \\
         &F_2(r, p_r, k)= \frac{L^{2}r^{-3} + (p_{z} -A )\partial_{r} A}{\sqrt{1+(p_{z} -A)^{2}+p_{r} ^{2} +L^{2} r^{-2} } } .
    \end{aligned}
\end{equation}
Hence \eqref{eq07} becomes
\begin{equation}\label{eq017}
    \begin{cases}
         \dot{r} = F_1(r, p_r, k) ,\\
         \dot{p}_r = F_2(r, p_r, k) .
    \end{cases}
\end{equation}
Moreover, by the regular perturbation method, \eqref{eq017} can be expanded as
    \begin{equation}\label{eq18}
         \begin{cases}
         \dot{r} = F_1(r, p_r, 0) + k \cdot \frac{\partial F_1}{\partial k}(r, p_r, 0) + O(k^2) , \\
         \dot{p}_r = F_2(r, p_r, 0) + k \cdot \frac{\partial F_2}{\partial k}(r, p_r, 0) + O(k^2) ,
         \end{cases}
    \end{equation}
where the first-order perturbation coefficients are
    \begin{equation*}
             \frac{\partial F_1}{\partial k}(r, p_r, 0) = -\frac{\frac{\mu_0}{2\pi} p_r\left( p_z + \frac{\mu_0}{2\pi} I_0 \ln r \right) a(t, r)}{\left( 1 + \left( p_z + \frac{\mu_0}{2\pi} I_0 \ln r \right)^2 + p_r^2 + L^2 r^{-2} \right)^{\frac{3}{2}}} , 
    \end{equation*}
and    
    \begin{multline*}
            \frac{\partial F_2}{\partial k}(r, p_r, 0) = -\frac{\left( \frac{\mu_0}{2\pi} \right)^2 I_0 r^{-1} a(t, r) + \frac{\mu_0}{2\pi} \left( p_z + \frac{\mu_0}{2\pi} I_0 \ln r \right)  \partial_r a(t, r)}{\sqrt{1 + \left( p_z + \frac{\mu_0}{2\pi} I_0 \ln r \right)^2 + p_r^2 + L^2 r^{-2}}}  \\
            - \frac{\left[ L^2 r^{-3} - \frac{\mu_0}{2\pi} I_0 r^{-1} \left( p_z + \frac{\mu_0}{2\pi} I_0 \ln r \right) \right] \frac{\mu_0}{2\pi} \left( p_z + \frac{\mu_0}{2\pi} I_0 \ln r \right) a(t, r)}{\left( 1 + \left( p_z + \frac{\mu_0}{2\pi} I_0 \ln r \right)^2 + p_r^2 + L^2 r^{-2} \right)^{\frac{3}{2}}} .
    \end{multline*}
The unperturbed system of equation \eqref{eq18} is
\begin{equation}\label{eq09}
   \begin{cases}
  & \dot{r}=\frac{p_{r} }{\sqrt{1+(p_{z} +\frac{\mu_0}{2\pi}I_0 \ln r)^{2}+p_{r} ^{2} +L^{2} r^{-2}    } },   \\
  & \dot{p_{r} }=\frac{L^{2}r^{-3} -\frac{\mu_0}{2\pi}I_0 r^{-1} (p_{z} +\frac{\mu_0}{2\pi}I_0 \ln r )}{\sqrt{1+(p_{z} +\frac{\mu_0}{2\pi}I_0 \ln r)^{2}+p_{r} ^{2} +L^{2} r^{-2} } }.
\end{cases}
\end{equation}
System \eqref{eq09} is also a Hamiltonian system, whose Hamiltonian energy function is
\begin{equation}\label{eq10}
    H(r,p_r) = \sqrt{1+(p_{z}+\frac{\mu _{0} }{2 \pi } I_{0} \ln{r})^{2} +p_r^2 + L^{2} r^{-2}}. 
\end{equation}

A special class of solutions, referred to as radial periodic solutions and defined as follows, is investigated in this study:

\begin{definition}
    A solution $q(t) = \left(r(t), \theta(t), z(t)\right)$ of \eqref{eq04} with linear momentum $p_z$ and angular momentum $L$ is called a $(L, p_z)$-solution. Moreover, it is \textit{radially $T$-periodic} if
    \begin{equation*}
        r(t + T) = r(t),\quad \text{for every}\quad t \in \mathbb{R}.
    \end{equation*}
\end{definition}

\section{The dynamics of the unperturbed system}\label{sec4}
The unperturbed system \eqref{eq09}, as represented by equation \eqref{eq18} at $k = 0$, is investigated in this section. In this case, the electric field vanishes, leading to a magnetostatic regime. 

For any pair $\left ( L,p_{z} \right ) \in \mathbb{R}^+\times \mathbb{R}$, it can be shown that there exists a unique equilibrium $\left ( \bar{r},0 \right )$ in \eqref{eq09}, which satisfies the equality
\begin{equation}\label{eq11}
      p_{z} +\frac{\mu _{0} }{2 \pi } I_{0} \ln{\bar{r}}  =\frac{2 \pi }{\mu _{0} }\frac{L^{2} }{I_{0}} \bar{r} ^{-2}.
\end{equation}
In order to analyze the stability of the equilibrium $(\bar{r}, 0)$, consider the linearized system of \eqref{eq09} with the linearized matrix
\begin{equation*}
    \begin{bmatrix}
  0 & \frac{1}{\sqrt{1+\left ( p_{z} +\frac{\mu _{0}}{2\pi } I_{0} \ln{\bar{r} } \right ) ^{2} +L^{2}\bar{r}^{-2} } } \\
  \frac{-2L^{2}\bar{r}^{-4}-\left ( \frac{\mu _{0}}{2\pi } I_{0} \right )^{2}\bar{r}^{-2} }{\sqrt{1+\left ( p_{z} +\frac{\mu _{0}}{2\pi } I_{0} \ln{\bar{r} } \right ) ^{2} +L^{2}\bar{r}^{-2} } }  & 0
\end{bmatrix},
\end{equation*}
the characteristic equation associated with this matrix is
\begin{equation*}
  \lambda ^{2}- \frac{-2L^{2}\bar{r}^{-4}-\left ( \frac{\mu _{0}}{2\pi } I_{0} \right )^{2}\bar{r}^{-2} }{1+\left ( p_{z} +\frac{\mu _{0}}{2\pi } I_{0} \ln{\bar{r} } \right ) ^{2} +L^{2}\bar{r}^{-2} } =0 .
\end{equation*}
It is straightforward to deduce that the characteristic roots are a pair of pure imaginary numbers. It follows that the equilibrium $\left( \bar{r}, 0 \right)$ is identified as a center of equation \eqref{eq09}, and it is a global center since the system is Hamiltonian. Additionally, due to the vector fields corresponding to the functions on the right-hand side of equation \eqref{eq09} being symmetric with respect to the $r$-axis, the solution orbits exhibit symmetry about the $r$-axis as well. This symmetry implies that the dynamics of the system are invariant under reflections about the $r$-axis. In particular, in this context, every solution of the system \eqref{eq04} is radially periodic and there exists a unique radial equilibrium.

In the following, the relationship between the periods of the periodic orbits of equation \eqref{eq09} and their respective Hamiltonian energy is established. Given that equation \eqref{eq09} is Hamiltonian, the Hamiltonian function $H(r, p_r)$ remains constant along each orbit, which is a fundamental property of Hamiltonian systems. Notably, each non-trivial orbit intersects the $r$-axis at two distinct points, denoted as $(r_a, 0)$ and $(r_b, 0)$, where $0<r_a<\bar{r}<r_b$ and $\bar{r}$ was defined in \eqref{eq11}. For a given Hamiltonian energy $H$, $r_a$ and $r_b$ are given by
\begin{equation*}
    H^2 = 1+(p_{z}+\frac{\mu _{0} }{2 \pi } I_{0} \ln{r})^{2} + L^{2} r^{-2} . 
\end{equation*}
From equation \eqref{eq10}, it can be deduced that
\begin{equation*}
    p_r = \sqrt{H^2-\left(1+(p_{z}+\frac{\mu _{0} }{2 \pi } I_{0} \ln{r})^{2} + L^{2} r^{-2}\right)} , 
\end{equation*}
Substituting this equation into the system \eqref{eq09} yields that
\begin{equation*}
    \frac{\mathrm{d} r}{\mathrm{d} t}= \frac{\sqrt{H^{2} -1-(p_{z}+\frac{\mu _{0} }{2 \pi } I_{0} \ln{r})^{2} - L^{2} r^{-2}}}{H} .
\end{equation*}
By integrating this equation along the path from $r_a$ to $r_b$, the periods \(T \) of the orbits with Hamiltonian \(H \) can be determined as
\begin{equation}\label{eq12}
    T(H)=2 \int_{r_a}^{r_b} \frac{H}{\sqrt{H^{2} -1-(p_{z}+\frac{\mu _{0} }{2 \pi } I_{0} \ln{r})^{2} - L^{2} r^{-2}}}\mathrm{d}r.
\end{equation}
The focus of the subsequent research is to determine whether different Hamiltonian energy correspond to different periods. To achieve this, it is essential to examine the derivative of the period with respect to the Hamiltonian energy. Prior to this examination, a simplifying function is introduced to facilitate the analysis, defined as
\begin{equation}\label{eq013}
    f(r) = 1 + \left(p_z + \frac{\mu_0}{2\pi} I_0 \ln r\right)^2 + L^2 r^{-2}.
\end{equation}
Thus, \eqref{eq12} can be reformulated as
\begin{equation*}
    T(H) = 2 \int_{r_a}^{r_b} \frac{H}{\sqrt{H^2 - f(r)}} \, \mathrm{d}r.
\end{equation*}
The derived expression for the period \( T(H) \) offers a more tractable form for further computation. With \( \bar{r} \) identified as the minimum point for the function \( f(r) \), the value \( H_0^2 \) is designated to represent the minimum value of \( f(r) \) at this point. Consequently, \( T(H) \) can be articulated as
\begin{equation*}
    T(H) = 2 \int_{r_a}^{r_b} \frac{H/\sqrt{H^2 - H_0^2}}{\sqrt{1 - \frac{f(r) - H_0^2}{H^2 - H_0^2}}} \, \mathrm{d}r.
\end{equation*}
Introducing the variables
\begin{equation*}
    \begin{aligned}
        &y^2 = \frac{f(r) - H_0^2}{H^2 - H_0^2}, \\
        &z = y \sqrt{H^2 - H_0^2}.
    \end{aligned}
\end{equation*}
Then $T(H)$ can be rewritten  as
\begin{equation*}
    T(H) = 2 \int_{-1}^{1} \frac{H/\sqrt{H^2 - H_0^2}}{\sqrt{1 - y^2}}\frac{\mathrm{d} r}{\mathrm{d} z} \frac{\mathrm{d} z}{\mathrm{d} y}  \, \mathrm{d}y.
\end{equation*}
By the definition of $y$ and $z$, it follows that 
\begin{equation*}
    \frac{\mathrm{d} z}{\mathrm{d} y} = \sqrt{H^2 - H_0^2},
\end{equation*}
which leads to
\begin{equation}\label{eq13}
    T(H) = 2 \int_{-1}^{1} \frac{H}{\sqrt{1 - y^2}}\frac{\mathrm{d} r}{\mathrm{d} z}  \mathrm{d}y.
\end{equation}
Since
\begin{equation*}
    z^2=1 + \left(p_z + \frac{\mu_0}{2\pi} I_0 \ln r\right)^2 + L^2 r^{-2} -H_0^2,
\end{equation*}
we introduce the function $g(r)$, satisfying
\begin{equation*}
    g(r)=\frac{1}{2} z^2.
\end{equation*}
Moreover, for each $y \in (-1, 1)$, define the variables $r_1$ and $r_2$ that satisfy $r_1 < \bar{r} < r_2$ and
\begin{equation}\label{eq14}
    y^2(H^2 - H_0^2) = 2g(r),
\end{equation}
where $\bar{r}$ was defined in \eqref{eq11}. The function $g(r)$ reaches its minimum at $\bar{r}$ and is monotonic on each side of $\bar{r}$. Prior to the proof of the theorem, the following lemmas are established.
\begin{lemma}\label{lem1}
    Function
\begin{equation*}
    s(r)=\frac{(g'(r))^2 -2g''(r) g(r)}{(g'(r))^3},
\end{equation*}
increases monotonically for $r \in (0,\infty)$.
\end{lemma}

\begin{lemma}\label{lem3}
    Let $T_0$ be the minimum period of the periodic solutions of \eqref{eq09}, then 
    \begin{equation*}
        T_0=\sqrt{\frac{2 \pi^2 H_0^2}{\left(\frac{\mu_0}{2\pi}I_0-p_z-\frac{\mu_0}{2\pi}I_0\ln \bar{r}\right)\frac{\mu_0}{2\pi}I_0 \bar{r}^{-2}+3L^2 \bar{r}^{-4}}}.
    \end{equation*}
\end{lemma}
\begin{proof}
For each $y \in (-1, 1)$, let $r_1$, $r_2$ satisfy $r_1< \bar{r}< r_2$  and \eqref{eq14}. It follows from the definition of the function $g(r)$ that
\begin{equation*}
    \frac{\mathrm{d} r}{\mathrm{d} z} = \frac{z}{g'(r)}.
\end{equation*}
Consequently,
\[
\begin{aligned}
    \frac{1}{2} T(H) & = \int_{-1}^{1}  \frac{1}{\sqrt{1-y^2}} \frac{\mathrm{d} r}{\mathrm{d} z}  H \, dy \\
    &=\int_{-1}^{0} \frac{1}{\sqrt{1-y^2}} \frac{-\sqrt{2g(r_1)}}{g'(r_1)} H \, dy + \int_{0}^{1} \frac{1}{\sqrt{1-y^2}} \frac{\sqrt{2g(r_2)}}{g'(r_2)} H \, dy.
\end{aligned}
\]
It leads that
\[
\begin{aligned}
\frac{1}{2} T(H_0)=& \lim_{H \to H_0} \left(\int_{-1}^{0} \frac{-\sqrt{2g(r_1)}}{g'(r_1) \sqrt{1-y^2}} H \, dy + \int_{0}^{1} \frac{\sqrt{2g(r_2)}}{g'(r_2) \sqrt{1-y^2}} H  dy \right) \\
=&\int_{-1}^{0} \lim_{H \to H_0} \left( \frac{H}{\sqrt{1-y^2}}  \frac{-2\sqrt{g(r_1)}}{g'(r_1)} \right)  dy \\
&+ \int_{0}^{1}\lim_{H \to H_0} \left(\frac{H}{\sqrt{1-y^2}}  \frac{\sqrt{2g(r_2)}}{g'(r_2)} \right) dy \\
=& \int_{-1}^{0} \frac{H_0}{\sqrt{1-y^2}} \sqrt{\frac{2}{g''(\bar{r})}} \, dy + \int_{0}^{1} \frac{H_0}{\sqrt{1-y^2}} \sqrt{\frac{2}{g''(\bar{r})}} \, dy \\
=& \frac{\pi}{2} H_0 \sqrt{\frac{2}{g''(\bar{r})}} + \frac{\pi}{2} H_0 \sqrt{\frac{2}{g''(\bar{r})}} \\
=& \sqrt{2} \pi H_0 \sqrt{\frac{1}{g''(\bar{r})}}.
\end{aligned}
\]
\end{proof}

\begin{lemma}\label{lem4}
    Let $T(H)$ be defined as \eqref{eq12}, then 
    \begin{equation*}
        T(H) \to +\infty \quad \text{when} \quad H \to +\infty.
    \end{equation*}
\end{lemma}
\begin{proof}
It can be deduced from \eqref{eq12} that
\begin{equation*}
\begin{aligned}
    T(H)=&2 \int_{r_a}^{r_b} \frac{H}{\sqrt{H^{2} -1-(p_{z}+\frac{\mu _{0} }{2 \pi } I_{0} \ln{r})^{2} - L^{2} r^{-2}}}\mathrm{d}r \\
    >&2(r_b-r_a).
\end{aligned}
\end{equation*}
Therefore, according to the definition of $r_a$ and $r_b$, $T(H)$ tends to infinity as $H$ tends to infinity.    
\end{proof}

Henceforth, it is straightforward that
\begin{theorem}\label{thm1}
    Let $I_0$ be positive, $I_1(t) \in C^2([0, T_1]; \mathbb{R})$ satisfying \eqref{eq05}. For any pair $(L, p_z) \in \mathbb{R^+} \times \mathbb{R}$, the periodic function $T(H)$ increases strictly monotonically with the Hamiltonian energy $H$, and $\lim_{H \to \infty}T(H)=\infty$.
\end{theorem}
\begin{proof}
Differentiating $T(H)$ with respect to $H$ in \eqref{eq13} yields
\begin{equation}\label{eq16}
    \frac{\mathrm{d}T(H)}{\mathrm{d}H} = 2 \int_{-1}^{1} \frac{1}{\sqrt{1 - y^2}} \left( z \frac{H^2}{H^2 - H_0^2}  \frac{\mathrm{d}^2r}{\mathrm{d}z^2} +  \frac{\mathrm{d}r}{\mathrm{d}z} \right) \mathrm{d}y.
\end{equation}
According to the definition of $g(r)$, it follows that
\begin{equation*}
    z\mathrm{d}z = g'(r)\mathrm{d}r.
\end{equation*}
This leads to 
\begin{equation*}
    \mathrm{d}z = g'(r)\mathrm{d}\left(\frac{\mathrm{d}r}{\mathrm{d}z}\right)+g''(r) \mathrm{d}r\left(\frac{\mathrm{d}r}{\mathrm{d}z} \right).
\end{equation*}
Dividing both sides of the equation by $\mathrm{d}z$ yields
\begin{equation*}
    1 = g'(r)\frac{\mathrm{d}^2r}{\mathrm{d}z^2}+g''(r)\left(\frac{\mathrm{d}r}{\mathrm{d}z} \right)^2.
\end{equation*}
Given that
\begin{equation*}
    \frac{\mathrm{d} r}{\mathrm{d} z} = \frac{z}{g'(r)},
\end{equation*}
it follows that
\begin{equation*}
    \frac{\mathrm{d}^2r}{\mathrm{d}z^2} = \frac{1 - g''(r) \left( \frac{\mathrm{d}r}{\mathrm{d}z} \right)^2}{g'(r)}.
\end{equation*}
Consequently, some terms in equation\eqref{eq16} become
\begin{equation*}
\begin{aligned}
    z \frac{H^2}{H^2 - H_0^2}  \frac{\mathrm{d}^2r}{\mathrm{d}z^2} +  \frac{\mathrm{d}r}{\mathrm{d}z}=& z \frac{H^2}{H^2 - H_0^2}  \frac{1 - g''(r) \left( \frac{\mathrm{d}r}{\mathrm{d}z} \right)^2}{g'(r)} +  \frac{z}{g'(r)}\\
    =& \frac{\mathrm{d}r}{\mathrm{d}z}\left(\frac{H^2}{H^2 - H_0^2} \left( 1 - g''(r) \left( \frac{\mathrm{d}r}{\mathrm{d}z} \right)^2 \right)+  1\right).
\end{aligned} 
\end{equation*}
Substituting this into \eqref{eq16}, it is derived that
\begin{equation*}
\begin{aligned}
    \frac{\mathrm{d}T(H)}{\mathrm{d}H} =& 2 \int_{-1}^{1} \frac{1}{\sqrt{1 - y^2}} \frac{\mathrm{d}r}{\mathrm{d}z}\left(\frac{H^2}{H^2 - H_0^2} \left( 1 - g''(r) \left( \frac{\mathrm{d}r}{\mathrm{d}z} \right)^2 \right)+  1\right) \mathrm{d}y \\
    =&2\int_{-1}^{0} \frac{1}{\sqrt{1 - y^2}} \frac{z}{g'(r)}\left(\frac{H^2 \left( 1 - g''(r_1) \left( \frac{\mathrm{d}r}{\mathrm{d}z} \right)^2 \right)}{H^2 - H_0^2} +  1\right) \mathrm{d}y \\
    &+2\int_{0}^{1} \frac{1}{\sqrt{1 - y^2}} \frac{z}{g'(r)}\left(\frac{H^2 \left( 1 - g''(r_2) \left( \frac{\mathrm{d}r}{\mathrm{d}z} \right)^2 \right)}{H^2 - H_0^2} +  1\right) \mathrm{d}y \\
    =&2\int_{0}^{1} \frac{1}{\sqrt{1 - y^2}} \frac{-y \sqrt{H^2 - H_0^2}}{g'(r)}\left(\frac{H^2 \left( 1 - g''(r_1) \left( \frac{\mathrm{d}r}{\mathrm{d}z} \right)^2 \right)}{H^2 - H_0^2} +  1\right) \mathrm{d}y \\
    &+2\int_{0}^{1} \frac{1}{\sqrt{1 - y^2}} \frac{y \sqrt{H^2 - H_0^2}}{g'(r)}\left(\frac{H^2 \left( 1 - g''(r_2) \left( \frac{\mathrm{d}r}{\mathrm{d}z} \right)^2 \right)}{H^2 - H_0^2} +  1\right) \mathrm{d}y \\    
    =&2\int_{0}^{1} \frac{1}{\sqrt{1 - y^2}} \frac{-\sqrt{2g(r_1)}}{g'(r_1)}\left(\frac{H^2 \left( 1 - g''(r_1) \frac{2g(r_1)}{\left(g'(r_1)\right)^2} \right)}{H^2 - H_0^2} +  1\right) \mathrm{d}y \\
    &+2\int_{0}^{1} \frac{1}{\sqrt{1 - y^2}} \frac{\sqrt{2g(r_2)}}{g'(r_2)}\left(\frac{H^2 \left( 1 - g''(r_2) \frac{2g(r_2)}{\left(g'(r_2)\right)^2} \right)}{H^2 - H_0^2} +  1\right) \mathrm{d}y.
\end{aligned} 
\end{equation*}
The following objective is to demonstrate that this integral above is positive, thereby obtaining the monotonicity of the period $T$ with respect to the Hamiltonian $H$. It suffices to show that
\begin{multline*}
    \frac{-\sqrt{2g(r_1)}}{g'(r_1)} \left( \frac{H^2}{H^2 - H_0^2} \left( 1 - g''(r_1) \frac{2g(r_1)}{(g'(r_1))^2} \right) + 1 \right)  \\+ \frac{\sqrt{2g(r_2)}}{g'(r_2)} \left( \frac{H^2}{H^2 - H_0^2} \left( 1 - g''(r_2)  \frac{2g(r_2)}{(g'(r_2))^2} \right) + 1 \right) > 0.
\end{multline*}
With the application of Lemma \ref{lem1} and
\begin{equation*}
    s(\bar{r})=-\frac{g'''(\bar{r})}{3(g''(\bar{r}))^2},
\end{equation*}
it can be deduced that
\begin{equation*}
    -\frac{1}{g'(r_1)} \left( 1 - 2\frac{g''(r_1) g(r_1)}{(g'(r_1))^2} \right)  + \frac{1}{g'(r_2)} \left( 1 - 2\frac{g''(r_2) g(r_2)}{(g'(r_2))^2} \right) \geq 0.
\end{equation*}
Note that
\begin{equation*}
    \frac{-1}{g'(r_1)}+ \frac{1}{g'(r_2)}>0,
\end{equation*}
this assertion is substantiated. Combined with Lemma \ref{lem4}, it is clear that the theorem holds.
\end{proof}

\section{The dynamics of the perturbed system}\label{sec5}

In this section, the existence of the radially periodic solutions is analyzed by using the Melnikov method. See \cite{guckenheimer2013nonlinear} \cite{wigginsintroduction} \cite{yagasaki1996melnikov} for more details.

Before presenting the result, an assumption is made regarding the function $a(t,r)$. It is recognized that $a(t,r)$ is time-periodic with period $T_1$. Consequently, it is reasonable to assume that $a(t,r)$ satisfies \eqref{eq06} and
\begin{equation}\label{eq08}
    a(t,r) =D(r) sin(\omega_1t),
\end{equation}
where $\omega_1=\frac{2\pi}{T_1}$ and $D(r) \in C^2(\mathbb{R^+}; \mathbb{R})$. Subsequently, the following result is established:
\begin{theorem}\label{thm2}
    Let $I_0$ be positive, $I_1(t) \in C^2([0, T_1]; \mathbb{R})$ satisfying \eqref{eq05} and $T_1 > T_0$. Take $0<k \ll 1$ and $a(t,r)$ satisfies \eqref{eq08}. For any pair $(L, p_z) \in \mathbb{R^+} \times \mathbb{R}$, there exists at most countable radially periodic solutions of \eqref{eq04} with periods being integer multiples of $T_1$. Moreover, for each such period, there is at most one corresponding periodic solution.
\end{theorem}
\begin{proof}
Firstly, let us focus on the system
\begin{equation}\label{eq19}
     \begin{cases}
     \dot{r} = F_1(r, p_r, 0) + k \cdot \frac{\partial F_1}{\partial k}(r, p_r, 0) ,\\
     \dot{p}_r = F_2(r, p_r, 0) + k \cdot \frac{\partial F_2}{\partial k}(r, p_r, 0) ,
     \end{cases}
\end{equation}
where the functions $F_1(r, p_r, 0)$ and $F_2(r, p_r, 0)$ are defined in \eqref{eq006} with $k=0$. Due to \eqref{eq08}, introducing the functions $G_1(r, p_r, 0)$ and $G_2(r, p_r, 0)$, \eqref{eq19} can be rewritten as
\begin{equation*}
     \begin{cases}
     \dot{r} = F_1(r, p_r, 0) + k \cdot G_1(r, p_r, 0) \sin(\omega_1 t) ,\\
     \dot{p}_r = F_2(r, p_r, 0) + k \cdot G_2(r, p_r, 0) \sin(\omega_1 t) ,
     \end{cases}
\end{equation*}
where
\begin{equation*}
\begin{aligned}
    &\frac{\partial F_1}{\partial k}(r, p_r, 0) = G_1(r, p_r, 0) \sin(\omega_1 t) ,\\
    &\frac{\partial F_2}{\partial k}(r, p_r, 0) = G_2(r, p_r, 0) \sin(\omega_1 t) .
\end{aligned}   
\end{equation*}
Note that both of the perturbation coefficients are periodic functions with period $T_1$. Therefore, for sufficiently small $k$, the Melnikov method can be employed to analyze the periodic solutions of equation \eqref{eq19}. To simplify the use of notation, $k=0$ will be omitted in the following. Notice that $(\bar{r}, 0)$ is the global center of the unperturbed system \eqref{eq09} and that the minimum period of the periodic solutions around the equilibrium is $T_0$. Thus for a periodic solution $\Gamma_1$ of the unperturbed system \eqref{eq09} with period $T_1>T_0$ and Hamiltonian $H_1$, define the Melnikov function as
\begin{multline}\label{eq018}
    M(t_0) = \int_{0}^{T_1} [F_1(r_1, p_{r_1}) G_2 (r_1, p_{r_1})\\ - F_{2}(r_1, p_{r_1}) G_1 (r_1, p_{r_1})] \sin \left(\omega_1 (t-t_0)\right) \mathrm{d}t.
\end{multline}
If this function has simple zeros, the perturbed system will have a   periodic solution with period $T_1$, located close to $\Gamma_1$. After a simple calculation that yields
\begin{equation*}
    F_1G_2 - F_{2} G_1=-\frac{\left( \frac{\mu_0}{2\pi} \right)^2 I_0 {r_1}^{-1} D(r_1) + \frac{\mu_0}{2\pi} \left( p_z + \frac{\mu_0}{2\pi} I_0 \ln {r_1} \right)  \partial_r D(r_1) }{H_1^2} p_{r_1}.
\end{equation*}
Note that the Fourier transform of the above $T_1$-periodic function $F_1G_2 - F_{2} G_1$ is
\begin{equation*}
   F_1G_2 - F_{2} G_1 = a_0 + \sum_{k=1}^{\infty} a_k  \cos\left(k \omega_1 t\right) 
    + \sum_{k=1}^{\infty} b_k \sin\left(k \omega_1 t\right),
\end{equation*}
where
\begin{equation*}
    \begin{aligned}
        &a_0=\frac{1}{T_1}\int_{0}^{T_1} (F_1G_2 - F_{2} G_1) (r_1, p_{r_1}) \, dt, \\
        &a_k=\frac{2}{T_1} \int_{0}^{T_1} (F_1G_2 - F_{2} G_1) (r_1, p_{r_1})  \cos\left(k \omega_1 t\right)  \mathrm{d}t,\\
        &b_k=\frac{2}{T_1} \int_{0}^{T_1} (F_1G_2 - F_{2} G_1) (r_1, p_{r_1}) \sin\left(k \omega_1 t\right)  \mathrm{d}t,
    \end{aligned}
\end{equation*}
are the Fourier coefficients. Substituting this Fourier transform into \eqref{eq018}, then
\begin{equation*}
    \begin{aligned}
        M(t_0)&=\int_0^{T_1} \left[ a_1  \cos\left( \omega_1 t\right)  \sin(\omega_1 (t-t_0)) +b_1 \sin\left(\omega_1 t\right)\sin(\omega_1 (t-t_0)) \right] \mathrm{d}t \\
        &=\int_0^{T_1} \left[ -a_1 \sin(\omega_1 t_0) \cos^2\left( \omega_1 t\right) +b_1 \cos (\omega_1 t_0) \sin^2\left(\omega_1 t\right) \right] \mathrm{d}t \\
        &=-a_1 \sin(\omega_1 t_0) \frac{T_1}{2}+b_1 \cos (\omega_1 t_0) \frac{T_1}{2} \\
        &=\frac{T_1}{2}\sqrt{a_1^2+b_1^2} \sin(\omega_1 t_0 + \phi),
    \end{aligned}
\end{equation*}
where $\phi$ is a constant that satisfies
\begin{equation*}
    \sin \phi =\frac{b_1}{\sqrt{a_1^2+b_1^2}},\quad \cos \phi =-\frac{a_1}{\sqrt{a_1^2+b_1^2}}.
\end{equation*}
Therefore, the Melnikov function $M(t_0)$ will have simple zeros when
\begin{equation*}
    a_1^2+b_1^2 \neq 0.
\end{equation*}
Due to Theorem \ref{thm1} it follows that when $T_1 > T_0$, there exists a unique $T_1$-periodic solution of \eqref{eq09}. Hence it can be deduced that there exists a unique $T_1$-periodic solution of system \eqref{eq19} in this context. 

For a periodic solution $\Gamma_n$ of the unperturbed system \eqref{eq09} with period $nT_1$ and Hamiltonian $H_n$, where $n$ is a positive integer and $n > 1$, define the Melnikov function as
\begin{multline}\label{eq019}
    M_n(t_0) = \int_{0}^{nT_1} [F_1(r_n, p_{r_n}) G_2 (r_n, p_{r_n})\\ - F_{2}(r_n, p_{r_n}) G_1 (r_n, p_{r_n})] \sin \left(\omega_1 (t-t_0)\right) \mathrm{d}t.
\end{multline}
Similar to the calculation above, substituting the Fourier transform of the $nT_1$-periodic function $(F_1G_2 - F_{2} G_1)(r_n, p_{r_n})$ into \eqref{eq019}, then
\begin{equation*}
    \begin{aligned}
        M_n(t_0)
        %&=\int_0^{nT_1} \left[ a_n^{(n)}  \cos\left( \omega_1 t\right)  \sin(\omega_1 (t-t_0)) +b_n^{(n)} \sin\left(\omega_1 t\right)\sin(\omega_1 (t-t_0)) \right] \mathrm{d}t \\
        &=\int_0^{nT_1} \left[ -a_n^{(n)} \sin(\omega_1 t_0) \cos^2\left( \omega_1 t\right) +b_n^{(n)} \cos (\omega_1 t_0) \sin^2\left(\omega_1 t\right) \right] \mathrm{d}t \\
        &=-a_n^{(n)} \sin(\omega_1 t_0) \frac{nT_1}{2}+b_n^{(n)} \cos (\omega_1 t_0) \frac{nT_1}{2} \\
        &=\frac{nT_1}{2}\sqrt{{(a_n^{(n)})}^2+{(b_n^{(n)})}^2} \sin(\omega_1 t_0 + \phi^{(n)}),
    \end{aligned}
\end{equation*}
where $a_n^{(n)}$, $b_n^{(n)}$ are the Fourier coefficients that 
\begin{equation*}
    \begin{aligned}
        &a_n^{(n)}=\frac{2}{nT_1} \int_{0}^{nT_1} (F_1G_2 - F_{2} G_1) (r_1, p_{r_1})  \cos\left(\omega_1 t\right)  \mathrm{d}t,\\
        &b_n^{(n)}=\frac{2}{nT_1} \int_{0}^{nT_1} (F_1G_2 - F_{2} G_1) (r_1, p_{r_1}) \sin\left(\omega_1 t\right)  \mathrm{d}t,
    \end{aligned}
\end{equation*}
and $\phi^{(n)}$ is a constant that satisfies
\begin{equation*}
    \sin \phi^{(n)} =\frac{b_n^{(n)}}{\sqrt{(a_n^{(n)})^2+(b_n^{(n)})^2}},\quad \cos \phi^{(n)} =-\frac{a_n^{(n)}}{\sqrt{(a_n^{(n)})^2+(b_n^{(n)})^2}}.
\end{equation*}
Therefore, the Melnikov function $M_n(t_0)$ admits simple zeros when 
\begin{equation*}
    (a_n^{(n)})^2+(b_n^{(n)})^2 \neq 0,
\end{equation*}
and it follows that, in this context, there exists a periodic solution of system \eqref{eq19} with period $nT_1$, which is located close to $\Gamma_n$. Additionally, for periodic solutions of the unperturbed system with other periods, there exist corresponding Melnikov functions that are constantly equal to zero, indicating that the perturbed system lacks periodic solutions with those specific periods. Consequently, by invoking Lemma \ref{lem3}, Lemma \ref{lem4} and Theorem \ref{thm1}, it is established that system \eqref{eq19} possesses at most countable periodic solutions, with periods being integer multiples of \( T_1 \) and each integer multiple corresponding to at most one periodic solution.
     
On the other hand, since the dynamical system \eqref{eq19} has at most countable periodic solutions and these are hyperbolic, this implies that the system is structurally stable. Thus, the dynamical properties of systems \eqref{eq18} and \eqref{eq19} are consistent when $k$ is sufficiently small. Consequently, when the current period $T_1 > T_0$, system \eqref{eq07} also has at most countable periodic solutions, with periods being integer multiples of $T_1$. Furthermore, for each such period, there is at most one corresponding periodic solution.
\end{proof}

\begin{remark}
    The set of coefficients $p_z$ and $I_0$ that make the Melnikov function equal to zero in the proof of the above theorem is relatively small. When $0<k\ll 1$, for certain parameters pair $(L, p_z)$, there may exist integers $n$ such that periodic solutions with period $nT_1$ do not exist, but if a periodic solution does exist for system \eqref{eq07}, its period must be an integer multiple of $T_1$.
\end{remark}

\begin{remark}
    In the extended phase space that includes time $t$, a charged particle in the unperturbed system (constant current) with initial energy $H > H_0$ is constrained to move on the corresponding invariant cylinder.  When a small harmonic modulation is introduced, the majority of these invariant cylinders are destroyed. Only a countable family of such cylinders persists, each corresponding to a radial periodic solution whose period is an integer multiple $nT_1$ of the current's period, with at most one such cylinder for each integer $n \geq 1$. These surviving invariant cylinders are organized hierarchically: any two adjacent ones form an invariant region in the extended phase space. Consequently, a particle starting within such a region remains trapped between the corresponding cylinders for all time. This structure implies that the particle's motion is strictly organized into discrete energy shells, illustrating a resonant hierarchy induced by the periodic driving.
\end{remark}

\section{Conclusions}
In this paper, we investigate the motion of a charged particle under the electromagnetic field generated by an electrically neutral infinite long straight wire with a time-periodic oscillating current, where the electromagnetic field solves Maxwell's equations uniquely for the corresponding distributional current $J$. Based on the reduction of the Lorentz force equation \eqref{eq04} to a Hamiltonian system \eqref{eq07} with one degree of freedom in \cite{garzon2023periodic}, this analysis reveals the rich resonance structure that emerges under harmonic current modulation. Using the Melnikov method for a small perturbation parameter $k$, We have proven that the continuous family of invariant cylinders in the unperturbed system breaks down into a discrete, countable set of periodic orbits. Specifically, we establish the existence of exactly one harmonic radial periodic solution and show that any other radial periodic solution must have a period that is an integer multiple $nT_1$ of the driving current's period. Moreover, for each such period, there is at most one corresponding periodic solution.

From a dynamical perspective, these results reveal a strict resonant hierarchy. The persistence of these specific periodic orbits implies that the phase space retains a structured stability: particles initialized in the regions between these surviving invariant cylinders remain radially confined. This work thus offers a mathematically rigorous picture of how time-dependent electromagnetic perturbations organize charged particle motion into discrete energy shells.

\section*{Acknowledgements}
The authors are supported by the NSF of China(Nos.12090014,12031020).
\appendix
\section{Proof of Lemma 1}\label{secA1}

Differentiating $s(r)$ yields
\begin{equation*}
    s'(r)=\frac{-2g'''(r) g(r) g'(r)-3g''(r) \left(g'(r)\right)^2+6\left(g''(r)\right)^2g(r)}{(g'(r))^4}.
\end{equation*}
Since the denominator is always positive, it suffices to show that the numerator is non-negative:
\begin{equation*}
    -2g'''(r) g(r) g'(r)-3g''(r) \left(g'(r)\right)^2+6\left(g''(r)\right)^2g(r) \geq 0.
\end{equation*}
 Consider the function $f(r)$ defined in equation \eqref{eq013}. Under the substitution
\begin{equation*}
x = e^{\frac{p_z}{\frac{\mu_0}{2\pi} I_0}} r,
\end{equation*}
the function $f$ becomes
\begin{equation}\label{eq016}
f(x) = 1 + I \ln^2 x + K x^{-2},
\end{equation}
where
\begin{equation*}
I = \left(\frac{\mu_0}{2\pi} I_0\right)^2, \quad K = L^2 e^{\frac{2 p_z}{\frac{\mu_0}{2\pi} I_0}}.
\end{equation*}
Differentiating \eqref{eq016} gives
\begin{equation*}
f'(x) = 2 I\ln x \cdot x^{-1} - 2Kx^{-3}.
\end{equation*}
There exists a unique $x_0$ such that $f'(x_0) = 0$, which implies
\begin{equation*}
K = I x_0^2 \ln x_0 .
\end{equation*}
Substituting this expression for $K$ into \eqref{eq016} yields
\begin{equation*}
f(x) = 1+I \ln^2 x + Ix_0^2 \ln x_0 \cdot x^{-2}.
\end{equation*}
Consequently, the function $g(x)$ can be written as
\begin{equation*}
g(x) =\frac{I}{2} (\ln^2 x + x_0^2 \ln x_0 \cdot x^{-2} -  \ln^2 x_0 -  \ln x_0).
\end{equation*}
Let
\begin{equation*}
G(x) = 2I^{-1} g(x_0 x), \quad a =\ln x_0,
\end{equation*}
which simplifies to
\begin{equation*}
G(x) =  (\ln x + a)^2 +  a  x^{-2} -  a^2 -  a,
\end{equation*}
where $a>0$. Therefore, it is sufficient to prove
\begin{equation*}
    -2G'''(x) G(x) G'(x)-3G''(x) \left(G'(x)\right)^2+6\left(G''(x)\right)^2G(x)\geq 0.
\end{equation*}
Replace $\log x$ with a new variable, which may as well be denoted as $x$. After simplification and calculation, it suffices to prove that
\begin{align*}
&3 \left( a + e^{2x} \left( -a + 2ax + x^2 \right) \right) \left( 3a + e^{2x} \left( 1-a - x \right) \right)^2 \\
&- \left( a + e^{2x} \left( -a + 2ax + x^2 \right) \right) \left( -a + e^{2x} (x+a) \right) \left( -12a + e^{2x} (-3+2a+2x) \right) \\
&- 3 \left( -a + e^{2x} (x+a) \right)^2 \left( 3a + e^{2x} (1-a - x) \right)\geq 0,
\end{align*}
where $x \in \mathbb{R}$, $a > 0$. Define the above function as $P(x, a)$, then the function can be expressed as a cubic polynomial in $a$:
\begin{equation*}
    P(x,a) = C_3 (x) a^3 + C_2 (x) a^2 + C_1 (x) a + C_0 (x),
\end{equation*}
where the coefficients are given by
\begin{align*}
C_3 (x) &= e^{6x} (2x+2) + e^{4x} (-8x-10) + e^{2x} (30x+2) + 6, \\
C_2 (x) &= e^{6x} (5x^2 + x) + e^{4x} (-12x^2 + 6x - 12) + e^{2x} (15x^2 + 17x + 12), \\
C_1 (x) &= e^{6x} (4x^3 - x^2 + 3x - 3) + e^{4x} (-4x^3 + x^2 + 3x + 3), \\
C_0 (x) &= e^{6x} x^4.
\end{align*}
Since $a > 0$, the non-negativity of $P(x, a)$ follows if each coefficient $C_i(x) \geq 0$ for all $x \in \mathbb{R}$. The following establishes this for each coefficient.

To show $C_1(x) \geq 0$, after substituting $x \mapsto x/2$, it suffices to show
\[
f_1(x) = e^{x}(2x^3 - x^2 + 6x - 12) - 2x^3 + x^2 + 6x + 12 \geq 0.
\]
Differentiating $f_1(x)$ gives
\begin{align*}
	f_1'(x) &= e^{x}(2x^3 + 5x^2 + 4x - 6) - 6x^2 + 2x + 6, \\
	f_1''(x) &= e^{x}(2x^3 + 11x^2 + 14x - 2) - 12x + 2, \\
	f_1'''(x) &= e^{x}(2x^3 + 17x^2 + 36x + 12) - 12, \\
	f_1^{(4)}(x) &= e^{x}(2x^3 + 23x^2 + 70x + 48).
\end{align*}
By analyzing the sign of $f_1^{(4)}(x)$, it can be determined that $f_1'''(x)>0$ for $x>0$ and $f_1'''(x)<0$ for $x<0$. Hence $f_1''(x) \geq f_1''(0) =0$. Furthermore, since $f_1'(0) = 0$, $f_1(x)$ attains its minimum at $x = 0$. Given that $f_1(0) = 0$, it follows that $f_1(x) \geq 0$ for all $x \in \mathbb{R}$, which implies $C_1(x) \geq 0$.

To prove $C_2(x) \geq 0$, substitute $x \mapsto x/2$ and define
\[
f_2(x) = e^{2x}(5x^2 + 2x) - 12e^{x}(x^2 - x + 4) + 15x^2 + 34x + 48.
\]
Differentiating $f_2(x)$ yields
\begin{align*}
	f_2'(x) &= e^{2x}(10x^2 + 14x + 2) - 12e^{x}(x^2 + x + 3) + 30x + 34, \\
	f_2''(x) &= e^{2x}(20x^2 + 48x + 18) - 12e^{x}(x^2 + 3x + 4) + 30, \\
	f_2'''(x) &= 4e^{x} \left( e^{x}(10x^2 + 34x + 21) - 3x^2 - 15x - 21 \right).
\end{align*}
Let $g_2(x) = e^{x}(10x^2 + 34x + 21) - 3x^2 - 15x - 21$. Then
\begin{align*}
	g_2'(x) &= e^{x}(10x^2 + 54x + 55) - 6x - 15, \\
	g_2''(x) &= e^{x}(10x^2 + 74x + 109) - 6, \\
	g_2'''(x) &= e^{x}(10x^2 + 94x + 183).
\end{align*}
Obviously, $g_2'''(x)$ has two real zeros. By examining the sign at these zeros, it can be deduced that $g_2''(x) > 0$ for $x > 0$ and $g_2''(x) < 0$ for $x < 0$. Hence $g_2'(x) \geq g_2'(0) >0$, which follows that $g_2(x)$ is strictly increasing and has a unique zero at $x=0$. Therefore $f_2''(x) \geq f_2''(0)=0$, thus $f_2(x)$ is convex and attains its unique minimum at $x=0$. Since $f_2(0) = 0$, it follows that $f_2(x) \geq 0$, and thus $C_2(x) \geq 0$.

To show $C_3(x) \geq 0$, define
\[
f_3(x) = e^{3x}(x+2) + e^{2x}(-4x-10) + e^{x}(15x+2) + 6.
\]
Its derivative is
\[
f_3'(x) = e^x \left( e^{2x}(3x+7) + e^{x}(-8x-24) + 15x+17 \right).
\]
Let $g_3(x) = e^{2x}(3x+7) + e^{x}(-8x-24) + 15x+17$. Then
\[
g_3'(x) = e^{2x}(6x+17) + e^{x}(-8x-32) + 15,
\]
and
\[
g_3''(x) = 4e^x \left( e^{x}(3x+10) - (2x+10) \right).
\]
Define $h_3(x) = e^{x}(3x+10) - 2x - 10$. Then
\[
h_3'(x) = e^{x}(3x+13) - 2, \quad h_3''(x) = e^{x}(3x+16).
\]
By analyzing the sign of $h_3''(x)$, it follows that $h_3'(x)$  has a unique real zero. This allows for a straightforward determination of the sign variation of $h_3(x)$, leading to the conclusion that $g_3'(x) \geq 0$ for all $x \in \mathbb{R}$, with equality only at $x=0$. Since $g_3(0)=0$, the function $f_3(x)$ attains its minimum at $x = 0$. Evaluating $f_3(0) = 0$ yields $f_3(x) \geq 0$ for all $x$, and therefore $C_3(x) \geq 0$.

Therefore, all coefficients $C_0(x)$, $C_1(x)$, $C_2(x)$, and $C_3(x)$ are non-negative for all $x \in \mathbb{R}$, which completes the proof. 

%%       \cite{<label>} ==> [1]

%%
%Example citation, See \cite{lamport94}.




%\begin{thebibliography}{00}

%% For numbered reference style

%\bibitem{lamport94}
%  \textit{\LaTeX: a document preparation system},
%  Addison Wesley, Massachusetts,

%\end{thebibliography}

\end{document}